\documentclass[reqno, 11 pt]{amsart}

\usepackage[a4paper, centering]{geometry}
\usepackage{amsmath,amssymb, amsthm,amsfonts}
\usepackage{mathabx} 
\usepackage[T1]{fontenc}
\usepackage{verbatim} 
\usepackage{enumerate}
\usepackage{xargs}

\theoremstyle{plain}
\newtheorem{theorem}{Theorem}
\newtheorem{proposition}[theorem]{Proposition}
\newtheorem{lemma}[theorem]{Lemma}
\newtheorem{corollary}[theorem]{Corollary}
\newtheorem{example}[theorem]{Example}
\newtheorem{notation}[theorem]{Notation}
\newtheorem{observation}[theorem]{Observation}

\theoremstyle{definition}
\newtheorem{definition}[theorem]{Definition} 
\newtheorem{remark}[theorem]{Remark}
\newtheorem{question}[theorem]{Question}

\newcommand{\eps}{\varepsilon}
\newcommand{\wbar}{\widebar}

\newcommand{\kk}{k}

\newcommand{\ZZ}{\mathbb{Z}}
\newcommand{\QQ}{\mathbb{Q}} 
\newcommand{\NN}{\mathbb{N}}
\newcommand{\GF}{\mathbb{F}}
\newcommand{\mf}{\mathfrak}
\newcommand{\cat}{\hat{\mathcal{C}}}
\newcommand{\rhobar}{\wbar{\rho}}

\newcommand{\ptor}[1]{ T_{p^\infty}(#1)}

\DeclareMathOperator{\Ann}{Ann}
\DeclareMathOperator{\im}{im}

\DeclareMathOperator{\rad}{rad}
\DeclareMathOperator{\nil}{nil}

\DeclareMathOperator{\hgt}{ht}

\DeclareMathOperator{\Hom}{Hom}

\DeclareMathOperator{\Char}{char}

\DeclareMathOperator{\Witt}{W}
\DeclareMathOperator{\GL}{GL}
\DeclareMathOperator{\SL}{SL}

\DeclareMathOperator{\Ob}{Ob}
\DeclareMathOperator{\Def}{Def}
\DeclareMathOperator{\Lift}{Lift}
\DeclareMathOperator{\Sets}{Sets}
\DeclareMathOperator{\Spec}{Spec}

\author{Krzysztof Dorobisz}
\title[A necessary condition for universal deformation rings of fin. group repr.]{A necessary condition for characteristic zero universal deformation rings of finite group representations}

\address{Mathematisch Instituut, Universiteit Leiden, P.O. Box 9512, 2300~RA Leiden, The Netherlands.}
\email{dorobiszkj@math.leidenuniv.nl}

\begin{document}

\begin{abstract}
We prove the following result related to the inverse problem for universal deformation rings of group representations:

Given a finite field $\kk$, denote by $W(\kk)$ the ring of Witt vectors over $\kk$ and by $K$ the field of fractions of $W(\kk)$. If a complete noetherian local ring $R$ is a universal deformation ring of a representation of a finite group over $\kk$, then $R \otimes_{W(\kk)} K$ is a finite \'etale $K$-algebra.
\end{abstract}

\maketitle

\section{Introduction}

This paper is related to author's earlier work \cite{Dorobisz}, in which the following result was proved: for every complete noetherian local commutative ring $R$ with a finite residue field $\kk$ there exist a \emph{profinite} group $G$ and a continuous linear representation $\rhobar: G \rightarrow \GL_n(\kk)$ of which $R$ is the universal deformation ring. This solves the so called inverse problem for universal deformation rings of group representations, formulated in \cite{BleherChinburgDeSmit}. More precisely, the main result of \cite{Dorobisz} is that $R$ is the universal deformation ring of a natural representation of $\SL_n(R)$ over $\kk$, provided that $n\geq 4$. Similar results have been achieved independently by Eardley and Manoharmayum in \cite{EardleyManoharmayum}. 

In this paper we want to address a modification of the inverse problem and restrict to representations of groups that are finite. We show that, unlike in the general case, there exist some rings which do not occur as universal deformation rings in the restricted setting. Furthermore, we present a non-trivial necessary condition for characteristic zero universal deformation rings of finite group representations. 

I would like to express my gratitude to Jakub Byszewski for suggesting me the approach that I have developed in section \ref{subsDim>1} and to Fabrizio Andreatta, who has helped me simplify the proofs in section \ref{subsDim=1}. I am also thankful to Hendrik Lenstra, Bart de Smit and Bas Edixhoven, discussions with whom have helped me shape this paper. 

\section{Notation and conventions}

The notation of the paper follows the one introduced in \cite{Dorobisz}. We briefly recall the most useful definitions, for more details see the mentioned paper. For general introduction to deformations of group representations see \cite{Mazur2}, \cite{BleherFinite} or \cite{Gouvea}. 

\subsection{Complete noetherian local rings}
In this paper $\kk$ will denote a finite field, $p$ its characteristic and $\Witt(\kk)$ the ring of Witt vectors over~$\kk$. Fixing $\kk$ we consider the category $\cat$ of all complete noetherian local commutative rings with residue field $\kk$. Each of them has a natural structure of a $\Witt(\kk)$-algebra and morphisms of $\cat$ are the local $\Witt(\kk)$-algebra homomorphisms. 

Given $R \in \Ob(\cat)$, we will denote the maximal ideal of this ring by $\mf{m}_R$. Cohen's structure theorem asserts that every $R \in \Ob(\cat)$ is a quotient of a power series ring $\Witt(\kk)[[X_1, \ldots, X_d]]$ for some $d$. Below we present a slightly less known result of Cohen, which can be viewed as an analog of E. Noether's normalization theorem.

\begin{theorem} \label{ThmIntExt}
Let $R \in \Ob(\cat)$ be such that either $\Char R = p$ or $\Char R = 0$ and $\hgt pR = 1$. Then there exists a subring $R_0$ of $R$ such that $R_0$ is isomorphic to a power series ring over $W(\kk)/ (\Char R)$ and $R$ is a finite $R_0$-module.
\end{theorem}
\begin{proof}
See \cite[Theorem 16]{Cohen}.
\end{proof}

Note that the condition $\hgt pR = 1$ is satisfied for example when $p$ is not a zero-divisor in $R$ (this is a consequence of Krull's ``Hauptidealsatz'').

\subsection{Deformations of group representations}

Let $G$ be a finite group, $n$ a positive integer and $\rhobar : G \rightarrow \GL_n(\kk)$ a representation. 

If $R \in \Ob(\cat)$ and $\pi$ is the reduction modulo $\mf{m}_R$, we define $\Lift_{\rhobar}(R) := \{ \rho : G \rightarrow \GL_n(R) \mid \rhobar = \pi \circ \rho \}$ and introduce the relation $\sim$ on $\Lift_{\rho}(R)$ as follows: $\rho \sim \rho'$ if and only if $\rho' = K \rho K^{-1}$ for some $K \in \ker \pi$. Subsequently, we define $\Def_{\rhobar}(R) := \Lift_{\rhobar}(R)/\!\!\sim$. This definition easily extends to a definition of a functor $\Def_{\rhobar} : \cat \rightarrow \Sets$. If $\Def_{\rhobar}$ is represented (in the sense of category theory) by $R \in \Ob(\cat)$, then we say that $R$ is the \emph{universal deformation ring} of $\rhobar$. 

The definition of the deformation functor $\Def_{\rhobar}$ becomes more elegant and better motivated using a module theoretic approach. Let $V_{\rhobar}$ be the representation space corresponding to the representation $\rhobar$. It is a $\kk G$-module of finite dimension over $\kk$. Given $R \in \Ob(\cat)$, we define a lift of $V_{\rhobar}$ to $R$ as a pair $(W, \iota)$, where $W$ is an $RG$-module, free and of finite rank over $R$, and $\iota : \kk \otimes_{R} W \cong V_{\rhobar}$ is an isomorphism of $\kk G$-modules. It is not difficult to check that the deformation functor $\Def_{\rhobar}$ classifies the (suitably defined) isomorphism classes of lifts of $V_{\rhobar}$; for details see \cite[\S 2.1]{BleherFinite}.

\subsection{Matrices}

Given a positive integer $n$ and an abelian group (a ring) $A$, we will denote by $M_n(A)$ the group (the ring) of $n \times n$ matrices with entries in $A$. Moreover, $I_n$ will stand for the identity matrix.

\section{Initial remarks and a cardinality argument} \label{SecInitial}

\begin{definition}
By $\mf{U}$ we will denote the class of all $R \in \Ob(\cat)$ for which there exists a finite group $G$ and a representation $\rhobar : G \rightarrow \GL_n(\kk)$ such that $R$ is the universal deformation ring of $\rhobar$. 
\end{definition}

We are interested in determining the class $\mf{U}$. To begin with, we observe that considerations of \cite{Dorobisz} lead to the following conclusion.

\begin{observation}
Every finite ring $R \in \Ob(\cat)$ belongs to $\mf{U}$.
\end{observation}
\begin{proof}
If $R \in \Ob(\cat)$ is finite then for every natural $n$ the group $\SL_n(R)$ is finite and by \cite[Theorem~1]{Dorobisz} we have $R \in \mf{U}$. 
\end{proof}

\noindent Moreover, some examples of infinite rings belonging to $\mf{U}$ can be found in the literature. These include:

\begin{itemize}
 \item $W(\kk)$, obtained for arbitrary representation of a group of order coprime to $p$; see \cite[Lemma~3]{Dorobisz},
 \item rings of the form $ W(\kk)[X_1, \ldots, X_m ] / ( X_1^{p^{k_i}}-1, \ldots, X_m^{p^{k_m}}-1)$, $m, k_1, \ldots, k_m \in \mathbb{N},$
 obtained for one-dimensional representations of finite groups; see \cite[\S 1.4]{Mazur2} or \cite[Proposition~3.13]{Gouvea},
\item $\ZZ_5[\sqrt{5}]$, obtained as one of the exceptional universal deformation rings in \cite[Proposition~24]{Dorobisz},
\end{itemize}

Finally, let us observe that a cardinality argument shows that contrary to the general case, there exist rings which can not be obtained as universal deformation rings in the restricted setting.

\begin{proposition} \label{PropCatUncount}
The class $\Ob(\cat)$ is uncountable.  
\end{proposition} 

\begin{proof}[Sketch of the proof]

We present an example of an uncountable family of pairwise non-isomorphic $\cat$-rings. For every $\alpha \in W(\kk)$ denote by $R_\alpha$ the following object of $\cat$: \[ R_\alpha := W(\kk)[X,Y]/\big	(\, (X,Y)^5 + (\, X^4, \ Y^4-X^2Y^2-\alpha X^3Y) \,\big).\] One checks by straighforward computations that $R_\alpha \cong_{\cat} R_\beta$ if and only if $\alpha = \pm \beta$. Since $W(\kk)$ is uncountable, this proves the claim.
\end{proof}

\begin{corollary} \label{CorUNotCat}
 The classes $\mf{U}$ and $\Ob(\cat)$ do not coincide.
\end{corollary}

\begin{proof}
There exist only countably many finite groups $G$ and each of them has only finitely many representations $\rhobar : G \rightarrow \GL_n(\kk)$ over the finite field $\kk$ (with $n \in \NN$ fixed). Consequently, $\mf{U}$ is at most countable, whereas $\Ob(\cat)$ is uncountable by Proposition~\ref{PropCatUncount}.
\end{proof}

We see that the interesting part of the modified inverse problem consists in determining which infinite rings belong to $\mf{U}$ and which do not. Since the argument of Corollary~\ref{CorUNotCat} is not constructive, a first step towards solving this problem is to provide concrete examples of $\cat$-rings not in $ \mf{U}$. This will occupy us in the rest of this paper.

\section{A motivating example} \label{SecMotivation}

Before developing a general approach, let us present a motivating example. In this section $R \in \Ob(\cat)$ will be a discrete valuation ring with a uniformizing element $\pi$ and a field of fractions $K$. The group $G$ will be finite and we will consider a residual representation $\rhobar: G \rightarrow \GL_n(\kk)$ about which we additionally assume that the centralizer of $\im \rhobar$ in $M_n(\kk)$ comprises only scalar matrices. This is not an unnatural assumption, since it is often used to ensure the existence of the universal deformation ring of $\rhobar$ (see \cite[Theorem~2.1]{BleherFinite}). 

\begin{proposition} \label{PropEquivOverK}
Lifts $\rho_1$, $\rho_2 \in \Lift_{\rhobar}(R)$ are strictly equivalent if and only if they are equivalent as representations over $K$. 
\end{proposition}
\begin{proof}
The ``only if'' part is obvious. Conversely, suppose there exists $A \in \GL_n(K)$ such that $\rho_1 = A \rho_2 A^{-1}$. Since $R$ is a discrete valuation ring, there exist $B \in M_n(R) \setminus \pi M_n(R)$ and $s \in \ZZ$ such that $A = \pi^s B$. Clearly $\rho_1 B = B \rho_2$. Reducing to $\kk$ we obtain $\rhobar \wbar{B} = \wbar{B} \rhobar$ and our assumption on $\im \rhobar$ implies that the image $\wbar{B}$ of $B$ is a (non-zero) scalar matrix. Therefore, there exist $u \in \mu_{R}$ and $B_0 \in I_n + M_n(\mf{m}_{R})$ such that $B = u B_0$. Since $\rho_1 = B_0 \,\rho_2\, B_0^{-1}$, it follows that $\rho_1$ and $\rho_2$ are strictly equivalent. 
\end{proof}

\begin{corollary} \label{CorFinitAd=kkI}
If $\Char R = 0$, then $\Def_{\rhobar}(R)$ is finite.
\end{corollary}
\begin{proof}
The field $K$ has characteristic zero, so the group ring $KG$ is semisimple by Maschke's theorem. Therefore, up to equivalence, there exist only finitely many $n$-dimensional representations of $G$ over $K$ and the claim follows from Lemma~\ref{PropEquivOverK}.
\end{proof}

\begin{corollary}
If $\Char R = 0$, then $R[[X]]$ is not a universal deformation ring of $\rhobar$.
\end{corollary}
\begin{proof}
Combine the preceding corollary with the fact that $\Hom_{\cat}(R[[X]], R)$ is infinite.
\end{proof} 

Our aim for the next sections is to generalize these observations. We will prove that the claim of Corollary~\ref{CorFinitAd=kkI} holds even without the extra assumption on $\im \rhobar$. This implies that $R[[X]] \notin \mf{U}$. We will develop an approach that gives a bound on the number of deformations to a wider class of rings of characteristic zero, not only to the discrete valuation rings. The bound, in turn, will allow us to formulate a nontrivial necessary condition for members of $\mf{U}$. Consequently, we obtain a large class of characteristic zero $\cat$-rings not belonging to $\mf{U}$.

\begin{remark}
Proposition~\ref{PropEquivOverK} applies also to discrete valuation rings of characteristic $p$. However, if $\Char K = \Char R = p$, there may exist infinitely many non-equivalent representations of $G$ over $K$ and Corollary~\ref{CorFinitAd=kkI} may not hold. As a result, in this case one can not draw a similar conclusion that $R[[X]] \notin \mf{U}$. 
\end{remark}

\section{Finiteness bounds}

\subsection{Main lemma}

\begin{definition}
We will denote by $\mf{F}$ the class of all rings $R \in \Ob(\cat)$ such that for every finite group $G$ and representation $\rhobar: G \rightarrow \GL_n(\kk)$ the set $\Def_{\rhobar}(R)$ is finite.
\end{definition}

It is clear that every finite $\cat$-ring is in $\mf{F}$, but what we are really interested in, is identifying some large subclass of infinite rings belonging to $\mf{F}$. This will be done using the following key lemma, inspired by \cite[Theorem~2]{Maranda}.

\begin{lemma}\label{LemMaranda}
Let $G$ be a finite group, $R \in \Ob(\cat)$ be a ring in which $|G|$ is not a zero-divisor and define  $J:= |G| \cdot \mf{m}_R \lhd R$. Then representations $\rho_1, \rho_2 : G \rightarrow \GL_n(R)$ are strictly equivalent if and only if their reductions $\pi_J \rho_1$ and $\pi_J \rho_2$ to $R/J$ are strictly equivalent.
 
\end{lemma}

\begin{proof}
The ``only if'' part of the lemma is obvious. For the ``if'' part note that $\pi_J \rho_1$ and $\pi_J \rho_2$ are strictly equivalent if and only if there exists $A \in I_n+M_n(\mf{m}_R)$ such that 
\[ \forall g\in G : \quad \rho_1(g) A \,\equiv\, A \rho_2(g) \mod{M_n(J)}. \]
If this is the case then $B := \sum_{g \in G} \rho_1(g) \, A \, \rho_2(g)^{-1}$ satisfies
$ B  \equiv |G| \cdot A \mod M_n(J).$ 
Using the definition of $J$ and the assumption that $|G|$ is not a zero divisor in $R$, we define $B_0 := \frac{1}{|G|} \cdot B \in M_n(R)$ and observe that $B_0 \in I_n + M_n(\mf{m}_R)$. Moreover,
\[  \rho_1(h)\, B_0\,\rho_2(h)^{-1} = \frac{1}{|G|}\cdot \sum_{g \in G} \rho_1(hg)\, A\, \rho_2 (hg)^{-1} = \frac{1}{|G|}\cdot \sum_{g \in G} \rho_1(g)\, A\, \rho_2(g)^{-1} = B_0,\]
 so $\rho_1$ and $\rho_2$ are strictly equivalent.
\end{proof}

\begin{remark} \label{RemCardinalityG}
In the setting of the above lemma let us write $|G| = p^rs$, $r\geq 0$, $p\nmid s$. Since $s$ is invertible in $R$, we have that $|G|$ is a zero-divisor if and only if $p^r$ is. Moreover, $|G|\mf{m}_R = p^r \mf{m}_R$. 

In particular, if $p \nmid |G|$ then the assumption that $|G|$ is not a zero-divisor in $R$ is satisfied for every $R \in \Ob(\cat)$ and the above lemma implies that there is at most one deformation to every $R \in \Ob(\cat)$. And actually there is exactly one deformation to every $R \in \Ob(\cat)$ -- see \cite[Lemma~3]{Dorobisz}.
\end{remark}

\begin{lemma} \label{LemDefiningW}
Consider $R \in \Ob(\cat)$, $r\in \mathbb{Z}_{\geq 1}$. The following conditions are equivalent:
\begin{enumerate}[(i)]
\item $p^r$ is not a zero-divisor in $R$ and $R / p^r\mf{m}_R$ is finite.
\item $p$ is not a zero-divisor in $R$ and $\dim R =1 $.  
\item $R$ is a finitely generated $W(\kk)$-module with trivial $p$-torsion part.  
\end{enumerate}
\end{lemma}
\begin{proof}
It is clear that $p^r$ is a zero-divisor if and only if $p$ is a zero-divisor. Since $\kk$ is finite we have that a ring $S \in \Ob(\cat)$ is finite if and only if it is Artinian, i.e., if and only if $\dim S = 0$. The fact that $p^r \mf{m}_R \subseteq pR \subseteq \rad (p^r \mf{m}_R)$ implies $\dim R/ p^r\mf{m}_R =\dim R/pR$. Assume that $p$ is not a zero-divisor. Then $\dim R/pR = \dim R - 1$, so $\dim R/p^r \mf{m}_R = 0$ if and only if $\dim R = 1$. This proves the equivalence of the first two statements. It is also clear that $(iii)$ implies $(ii)$. The converse statement follows from Theorem~\ref{ThmIntExt}.
\end{proof}

\begin{definition}
We will denote by $\mf{W}$ the subclass of all rings $R \in \Ob(\cat)$ satisfying the equivalent conditions of Lemma~\ref{LemDefiningW}. 
\end{definition}

\begin{corollary} \label{CorWinF}
If $R \in \mf{W}$ then $R \in \mf{F}$.
\end{corollary}
\begin{proof}
By the first property of Lemma~\ref{LemDefiningW} the ring $R/ n\mf{m}_R$ is finite for every $n \geq 1$ (see also Remark~\ref{RemCardinalityG}).
Since all finite $\cat$-rings are in $\mf{F}$, the claim follows easily from Lemma~\ref{LemMaranda}.
\end{proof}

The result obtained in Corollary~\ref{CorWinF} is fully satisfactory for our applications in the next sections, but we note that it can be further extended.

\begin{definition}
Given an abelian group $A$ and a prime number $p$ we will denote by $\ptor{A}$ its $p$-torsion subgroup, i.e., $\ptor{A} =  \bigcup_{r=1}^\infty \{a \in A \mid p^r a = 0\}.$
\end{definition}

\begin{observation} \label{LemRp}
If $R$ is a ring then $\ptor{R}$ is an ideal. Suppose that $\Char R = 0$ and let $\tilde{R} := R/\ptor{R}$. Then $\Char \tilde{R} = 0$ and $p$ is not a zero-divisor in $\tilde{R}$. 
\end{observation}

\begin{lemma} 
Let $R \in \Ob(\cat)$ be of characteristic zero and finitely generated as a $W(\kk)$-module. Then $R \in \mf{F}$.
\end{lemma} 
\begin{proof}
Let $G$ be a finite group and $\rhobar: G \rightarrow \GL_n(\kk)$ its representation. The ring $\tilde{R} := R/\ptor{R}$ is in $\mf{W}$ (Observation~\ref{LemRp} implies easily that $\tilde{R}$ satisfies the condition $(iii)$ of Lemma~\ref{LemDefiningW}), so by Corollary~\ref{CorWinF} the set $\Def_{\rhobar}(\tilde{R})$ is finite. Noetherianity of $R$ implies that there exists $r \geq 1$ such that $\ptor{R} = \Ann p^r$. Hence $\ptor{R}$ is a finite $W(\kk)/ (p^r)$-module and, consequently, a finite set. It follows that the fibers of the map $\Lift_{\rhobar}(R) \rightarrow \Lift_{\rhobar}(\tilde{R})$ induced by the reduction modulo $\ptor{R}$ are finite and so $\Def_{\rhobar}(R)$ is finite as well.
\end{proof}

\subsection{Properties of $\mf{W}$-rings}

\begin{notation}
In the rest of the paper we reserve the letter $K$ to denote the field of fractions of the ring $W(\kk)$.
\end{notation}

\begin{lemma} \label{LemR'}
Consider $R \in \mf{W}$ and its localization $R' := R[\frac{1}{p}]$ away from $p$. 
\begin{enumerate}[(i)]
\item The natural map $R \rightarrow R'$ is injective and $R$ is a domain (is reduced) if and only if $R'$ is a domain (is reduced). \label{LemR'Inj}
\item $R'$ is naturally isomorphic to $R \otimes_{\,W(\kk)} K$.
\item $R'$ is an integral extension of $K$.\label{LemR'Struct}
\item $R'$ is a domain if and only if it is a finite field extension of $K$. If this is the case then $R'$ is the field of fractions of $R$. \label{LemR'Dom}
\item $R'$ is reduced if and only if it is a finite product of finite field extensions of $K$. \label{LemR'Red}
\end{enumerate}
\end{lemma}
\begin{proof}
\begin{enumerate}[(i)]
\item This is an easy consequence of the fact that $p$ is not a zero-divisor in $R$.
\item Note that $K = W(\kk)[\frac{1}{p}]$ and use the identification $B_S \cong R \otimes_A A_S$, valid for any $A$-algebra $B$ and multiplicative subset $S \subseteq A$. 
\item The extension $W(\kk) \subseteq R$ is integral. Localizing away from $p$ we obtain that $K = W[\frac{1}{p}] \subseteq R[\frac{1}{p}] = R'$ is integral as well. 
\item The first claim follows from part (\ref{LemR'Struct}) and general properties of integral extensions (see for example \cite[Proposition 5.7]{Atiyah}). The second one is clear, since the fraction fields of domains $R$ and $R[\frac{1}{p}]$ coincide.

\item In general, part (\ref{LemR'Struct}) implies that $R'$ is artinian, so by the structure theorem (\cite[Theorem~8.7]{Atiyah}), there exist artinian local rings $A_1, A_2, \ldots A_s$ such that $R \cong A_1 \oplus A_2 \oplus \ldots \oplus A_s$. Note that these rings are necessarily integral extensions of $K$.

The ring $R'$ is reduced if and only if all $A_i$'s have this property. To finish the proof, observe that a local artinian ring $A$ is reduced if and only if it is a domain and use again \cite[Proposition 5.7]{Atiyah}.
\end{enumerate}
\end{proof}

\section{Necessary conditions for members of $\mf{U}$}

Keeping in mind Corollary~\ref{CorWinF} and the general idea outlined in section~\ref{SecMotivation}, we turn our attention to the following problem: for which $R \in \Ob(\cat)$ does there exist $S \in \mf{W}$ such that $\Hom_{\cat}(R,S)$ is infinite? 

\begin{observation}
If $S$ is a ring in which $p$ is not a zero-divisor then every ring homomorphism $R \rightarrow S$ factors via $R/\ptor{R}$.
\end{observation} 

The above observation implies that it is enough to solve the problem for $\cat$-rings $R$ in which $\ptor{R}=0$, i.e., in which $p$ is not a zero-divisor. 
Observe that such rings are of characteristic zero, hence infinite and of Krull dimension greater than zero.

In what follows $R$ is a $\cat$-ring, we assume that $\ptor{R}=0$ and set $d := \dim R$. By Theorem~\ref{ThmIntExt} there exists a subring $R_0 \subseteq R$, isomorphic to $W(\kk)[[X_1, \ldots, X_{d-1}]]$, over which $R$ is a finite module. 

We divide further discussion into two cases, depending on whether $d \geq 2$ or $d = 1$.

\medskip
\subsection{Case $\dim R \geq 2$} \label{subsDim>1} 

\begin{lemma} \label{LemCountablyManyDomains}
There exist (up to isomorphism) exactly countably many integral domains in $\mf{W}$ .
\end{lemma}
\begin{proof}

Since for every $n\in \mathbb{N}$ we have the integral domain $\ZZ_p[\sqrt[n]{p}\,] \in \mf{W}$, the interesting part of the proof consists in showing that  $\mf{W}$ contains at most countably many integral domains.

It is a classical fact, following from Krasner's lemma, that for every $n \in \NN$ the field $\QQ_p$ has (up to isomorphism) only finitely many extensions of degree $n$, see \cite[Theorem~4.8]{Stevenhagen}. As a consequence, there exist only countably many finite field extensions of $K$. Therefore, in view of Lemma~\ref{LemR'}.(\ref{LemR'Inj}) and (\ref{LemR'Dom}), it is sufficient to prove that every family of pairwise non-isomorphic domains $S \in \mf{W}$ with the same field of fractions $L$ is at most countable.

Let $L$ be a finite field extension of $K$ and consider $S \subseteq L$ such that $S \in \mf{W}$ and $L$ is the field of fractions of $S$. Since $S$ is integrally dependent on $W(\kk)$, it is contained in the integral closure of $W(\kk)$ in $L$, i.e., in the ring of integers $\mathcal{O}_L$ of $L$. On the other hand,  $\mathcal{O}_L$ is a finitely generated $W(\kk)$-module and by Lemma~\ref{LemR'}.(\ref{LemR'Dom}) we have $L = S[\frac{1}{p}]$, so there exists $r \in \mathbb{N}$ such that $p^r \mathcal{O}_L \subseteq S$. To finish the proof, it is sufficient to show that for every $r \in \mathbb{N}$ there exist only finitely many $W(\kk)$-modules $M$ such that $p^r \mathcal{O}_L \subseteq M \subseteq \mathcal{O}_L$. But this is obvious: such modules correspond bijectively with $W(\kk)$-submodules of $\mathcal{O}_L / p^r \mathcal{O}_L$, which is a finite set (indeed: it is a finitely generated $W(\kk)/(p^r)$-module).
\end{proof}

\begin{theorem}\label{ThmCaseDim>1}
If $d \geq 2$ then there exists an integral domain $S \in \mf{W}$ such that $\Hom_{\cat}(R, S)$ is infinite.
\end{theorem}
\begin{proof}
Due to the assumption $d \geq 2$, there exist uncountably many prime ideals $\mf{p}$ of $R_0$ with the property $R_0/\mf{p} \cong W(\kk)$ (recall that by definition $R_0 \cong W(\kk)[[X_1, \ldots, X_{d-1}]]$). If $\mf{p}$ is one of them then there exists a prime ideal $\mf{q}$ of $R$ such that $\mf{q} \cap R_0 = \mf{p}$ (\cite[Theorem~5.10]{Atiyah}). The domain $R/\mf{q}$ is an integral extension of $S/\mf{p} \cong W(\kk)$, hence belongs to $\mf{W}$. We obtain thus uncountably many surjections $R \rightarrow R/\mf{q}$ from $R$ to some integral domain in $\mf{W}$. Lemma~\ref{LemCountablyManyDomains} and infinite pigeonhole principle imply that for some integral domain $S \in \mf{W}$ the set $\Hom_{\cat}(R, S)$ is infinite (even uncountable).
\end{proof}

\begin{remark}
It is tempting to ``refine'' the above theorem by changing ``integral domain'' to ``discrete valuation ring'', using the following argument: 

\textit{If a domain $S \in \mf{W}$ has the field of fractions $L$ then $S \subseteq \mathcal{O}_L$. Let us thus compose the considered morphisms $R \rightarrow S$ with inclusions $S \hookrightarrow \mathcal{O}_L$, in order to obtain infinitely many morphisms $R \rightarrow \mathcal{O}_L$.}

However, $\mathcal{O}_L$ is not necessarily a $\cat$-ring. Even though it is complete, local and noetherian, its residue field may be strictly larger than $\kk$; see the example below. 
\end{remark}

\begin{example}  
Suppose $p$ is a prime number satisfying $p \equiv 3 \pmod{4}$, i.e., such that $-1$ is not a quadratic residue in $\GF_p$ and let $R:= \ZZ_p[[X,Y]]/(X^2+Y^2)$. 

Consider the integral domain $S := \ZZ_p[T]/(T^2+p^2) \in \mf{W}$. For every $a \in \ZZ_p$ we have a $\cat$-morphism defined by $X \mapsto aT$, $Y \mapsto ap$, so $\Hom_{\cat}(R,S)$ is infinite.

On the other hand, if $S \in \mf{W}$ is a discrete valuation ring then the only $a,b \in S$ for which $a^2+b^2 = 0$ are $a=b=0$. Hence, $\Hom_{\cat}(R,S)$ is a one-element set. 
\end{example}

\medskip
\subsection{Case $\dim R = 1$} \label{subsDim=1}

If $d= 1$ and $\ptor{R}=0$ then $R$ itself belongs to $\mf{W}$. It can be shown that $\Spec R$ is finite and so the approach of the preceding subsection can not be applied in this case. Yet there may still exist some $S \in \mf{W}$ for which the set $\Hom(R,S)$ is infinite. 

\begin{example} Let $R := W(\kk)[\eps]$. The only $\cat$-morphism $R \rightarrow S$ such that $S$ is a characteristic zero integral domain is the reduction $R \xrightarrow{\eps \mapsto 0} W(\kk)$, but for $S:= R$ there exist infinitely many morphisms $R \rightarrow S$. Indeed, for every $C \in W(\kk)$ we can define a $\cat$-homomorphism \[ W(\kk)[\eps] \, \ni \, x+ \eps y \ \longmapsto \ x + C \, \eps y \, \in \, W(\kk)[\eps]. \] 
Consequently, $R \notin \mf{U}$. 
\end{example}

Note that all infinitely many $\cat$-homomorphisms $R \rightarrow S$ constructed in the above example reduce to the same morphism modulo $I := \eps S$. Moreover, the ideal $I \lhd S$ has the property $I^2=0$. In the rest of this subsection we will construct pairs $(R,S)$ with similar properties. 

In what follows, we will need the notions of derivations and K\"ahler differentials. Their definitions and basic properties can be found for example in \cite[\S 16]{Eisenbud}.

\begin{proposition} \label{PropDeriv}
Let $A$ be a ring and $f : R \rightarrow S$ be a homomorphism of $A$-algebras. Suppose there exists an ideal $I \lhd S$ with property $I^2=0$ and let  $g : R \rightarrow S$ be an additive map such that $f \equiv g \pmod{I}$. Then $g$ is an $A$-algebra homomorphism if and only if the map $(f-g) : R \rightarrow I$ is an $A$-derivation.
\end{proposition}
\begin{proof}
Apply \cite[Proposition 16.11]{Eisenbud}.
\end{proof}

\begin{notation}
Given a ring $A$ and an $A$-algebra $R$ we will denote by $\Omega_{R/A}$ the $R$-module of relative K\"ahler differentials of $R$.
\end{notation}

\begin{lemma} \label{LemRedFin}
If $R,S \in \mf{W}$ and $S$ is reduced then $\Hom_{\cat}(R,S)$ is finite.
\end{lemma}
\begin{proof}
By Lemma~\ref{LemR'}.(\ref{LemR'Inj}) it is sufficient to prove that there exist only finitely many $W(\kk)$-algebra homomorphisms $R \rightarrow S[\frac{1}{p}] =: S'$. Let $x_1, \ldots, x_m \in R$ be such that $R = W(\kk)[x_1, \ldots, x_m]$. For every $i \in \{ 1,\ldots, m\}$ there exists a monic polynomial $F_i \in W(\kk)[X]$ of which $x_i$ is a root. If $S$ is reduced then $S'$ is a finite product of fields by Lemma~\ref{LemR'}.(\ref{LemR'Red}). Hence, each of these polynomials has only a finite number of roots in $S'$, so there exist only finitely many values in $S'$ to which $x_i$, $i \in \{ 1,\ldots, m\}$, cn be mapped. This proves the claim.
\end{proof}

\begin{theorem} \label{ThmCaseDim=1}
Consider $R \in \mf{W}$ and define $R' := R[\frac{1}{p}]$, $\Omega := \Omega_{R/W(\kk)}$, $\Omega' := \Omega_{R'/K}$. The following conditions are equivalent. 
\begin{enumerate}[(i)]
\item There exists $S \in \mf{W}$ such that $\Hom_{\cat}(R,S)$ is infinite.
\item $\Omega \neq \ptor{\Omega}$. 
\item $\Omega'$ is not trivial.
\item $R$ is not reduced.
\end{enumerate}
\end{theorem}

\begin{proof}

$(i) \Rightarrow (ii)$: 
Let $N := \nil S$ and observe that by noetherianity there exists $m \in \mathbb{N}$ such that $N^m=0$. We have a finite chain of morphisms 
\[S = S/N^m \twoheadrightarrow S/N^{m-1} \twoheadrightarrow \ldots \twoheadrightarrow S/N^2 \twoheadrightarrow S/N. \]
By assumption, $\# \Hom_{\cat}(R, S/N^m)=\infty$ and $\# \Hom_{\cat}(R, S/N) <\infty$ by Lemma~\ref{LemRedFin}. Hence, there exists $r \in \{1,2, \ldots, m-1\}$ such that 
\[ \# \Hom_{\cat}(R,S/N^{r+1}) = \infty  \quad \hbox{and} \quad \#\Hom_{\cat}(R,S/N^{r}) < \infty .\] Let $\tilde{S}:=S/N^{r+1}$, $I := N^r / N^{r+1}$. Then $\tilde{S} \in \mf{W}$ and $I$ is such a non-zero ideal of $\tilde{S}$ that $I^2=0$. By infinite pigeonhole principle, there exists an infinite family of $\cat$-morphisms $ R \rightarrow \tilde{S}$ with the same reduction modulo $I$. By Proposition~\ref{PropDeriv}, it corresponds to an infinite family of derivations $R \rightarrow I$, so $\Hom_R(\Omega,I)$ is infinite by the definition of $\Omega$. Since $p$ is not a zero-divisor in $\tilde{S}$, we have $\Hom_R(\Omega,I) = \Hom_R(\Omega/\ptor{\Omega},I)$. In particular, $\Omega/\ptor{\Omega}$ must be non-trivial.
 
 \smallskip
 
$(ii) \Rightarrow (i)$: Define $M := \Omega/ \ptor{\Omega}$ and let us adopt the convention of writing $[\omega]$ for the corresponding class of $\omega \in \Omega$ in $M$. Note that $\Omega$ is a finitely generated $R$-module (because $R$ is a finitely generated $W(\kk)$-algebra) and hence, so are $M$ and $S := R \oplus M$. 

We introduce the ring structure on $S$ using the scalar multiplication and setting $xy=0$ for every $x,y \in M$. The obtained ring is clearly local, complete and noetherian (here it is important that $M$ is a finitely generated $R$-module) and has the same residue field as $R$. Moreover, $\ptor{S}$ is trivial and since $R \subseteq S$  is an integral extension, we also have $\dim S = \dim R = 1$. Therefore, $S \in \mf{W}$. Furthermore, for every $C \in W(\kk)$ the map $R \ni x \mapsto x + C \cdot [dx] \in S$ is a well-defined $\cat$-morphism. Since $M$ is a (by assumption non-trivial) free $W(\kk)$-module, all these morphisms are pairwise distinct. We conclude that $S$ satisfies all the required properties. 
  
\smallskip

$(ii) \Leftrightarrow (iii)$: This part follows from $\Omega' =  \Omega_{R'/K} \cong \Omega_{R/W(\kk)} \otimes_{\,W(\kk)} K \cong \Omega[\frac{1}{p}]$ (note that formation of dirrentials commutes with base change: \cite[Proposition~16.4]{Eisenbud}) and an easy observation that $\Omega[\frac{1}{p}]=0$ if and only if $\Omega = \ptor{\Omega}$.

\smallskip

$(iii) \Leftrightarrow (iv)$: $R'$ is a finitely generated $K$-algebra, so $\Omega'=0$ if and only if $R'$ is a finite direct product of fields, each finite and separable over $K$ (\cite[Corollary~16.16]{Eisenbud}). Note that $\Char K=0$, so every field extension of $K$ is separable and combine this result with Lemma~\ref{LemR'}.(\ref{LemR'Red}).

\end{proof}

\section{The main result, corollaries and comments}

\subsection{Main result}

The following result is an immediate consequence of Corollary~\ref{CorWinF}, Theorem~\ref{ThmCaseDim>1} and Theorem~\ref{ThmCaseDim=1}:

\begin{theorem} \label{ThmCritUDR}
Let $R \in \Ob(\cat)$ be of characteristic zero and in $\mf{U}$. Then $R/\ptor{R}$ is reduced and of Krull dimension $1$.
\end{theorem}
\begin{proof}
If $R$ is a universal deformation ring of a representation $\rhobar$ of a finite group $G$ and $S \in \mf{W}$, then $\Def_{\rhobar}(S)$ is finite by Corollary~\ref{CorWinF}. Hence, so is $\Hom_{\cat}(R,S)$. Theorem~\ref{ThmCaseDim>1} implies thus that $\dim R/\ptor{R} = 1$ and Theorem~\ref{ThmCaseDim=1} implies that $R/\ptor{R}$ is reduced.
\end{proof}

Alternatively, we can phrase this theorem as follows:

\begin{theorem}
 If $R \in \mf{U}$ then $R \otimes_{W(\kk)} K$ is a finite \'etale $K$-algebra.
\end{theorem}

Note that the infinite rings belonging to $\mf{U}$ that were mentioned in Section~\ref{SecInitial} indeed satisfy the conditions of the theorem. On the other hand, we can use it to produce explicit examples of $\cat$-rings not belonging to $\mf{U}$.

\begin{example}
The rings $W(\kk)[[X]]$ and $W(\kk)[\eps]$ are not in $\mf{U}$.
\end{example}

\begin{remark}
Let us return to the construction described in the proof of Proposition~\ref{PropCatUncount}. Theorem~\ref{ThmCritUDR} implies that actually none of the constructed uncountably many rings belongs to $\mf{U}$. However, Proposition~\ref{PropCatUncount} still provides some extra information. Namely, it implies that there are uncountably many $\cat$-rings that can not be obtained even as \emph{versal} deformation rings (see for example \cite[\S 3.1]{Dorobisz}) of finite group representations. 
\end{remark}

\begin{remark}
Note that contrary to the case $R = W(\kk)[X]/(X^2)$, every ring of the form $W(\kk)[X]/(X^2,p^rX)$, $r \in \NN_{\geq 1}$, belongs to $\mf{U}$, as was shown by Bleher, Chinburg and de Smit in \cite{BleherChinburgDeSmit}. Theorem~\ref{ThmCritUDR} explains where does the main difference lie between these two similar cases.
\end{remark} 
	
	Observe also how easy it is to arrive at an unsolved case: the author is not aware of any result concerning the problem whether, given $r \in \NN_{\geq 1}$, the ring $W(\kk)[X]/(X^2-p^rX)$ belongs to $\mf{U}$ or not.

\subsection{Quotients of universal deformation rings}

\begin{definition}
Let us denote by $\mf{Q}$ the subclass of all $\cat$-rings of the form $R/I$, where $R \in \mf{U}$ and $I$ is its proper ideal.
\end{definition}

\begin{remark} \label{RemStrengtheningUtoQ}
 Theorem~\ref{ThmCritUDR} holds true also if we replace "in $\mf{U}$", by "in $\mf{Q}"$. Indeed, if $S \in \mf{W}$ and $R \in \mf{Q}$ is a quotient of $R' \in \mf{U}$, then finiteness of the set $\Hom_{\cat}(R',S)$ implies finiteness of the set $\Hom_{\cat}(R,S)$. The proof of Theorem~\ref{ThmCritUDR} is thus valid also in case $R \in \mf{Q}$. 
\end{remark}

One could hope that this strengthening would allow to obtain new results about the class $\mf{U}$ itself. That is, a priori, one could expect that Theorem~\ref{ThmCritUDR} does not exclude $R$ from being in $\mf{U}$, but some of its quotient is excluded from being in $\mf{Q}$ by Remark~\ref{RemStrengtheningUtoQ}. However, this is not the case, due to the following easy observation. 

\begin{observation}
Let $R \in \Ob(\cat)$ and $I \lhd R$ be such that $\Char R  = \Char (R/I) = 0$. If $R/\ptor{R}$ is both reduced and of dimension one, then so is $(R/I) /\ptor{R/I}$.
\end{observation}

Thus, the following problem remains open:

\begin{question}
Obviously $\mf{U} \subseteq \mf{Q}$. But is $\mf{Q}$ strictly larger than $\mf{U}$?
\end{question}

\subsection{Towards extending the main result} 

\begin{question}
Let $S \in \Ob(\cat)$ be a one-dimensional, reduced ring in which $p$ is not a zero-divisor (in particular: of characteristic zero). 
\begin{enumerate}[$(i)$]
\item Does $S \in \mf{Q}$ hold? 
\item If $S \in \mf{Q}$, which rings $R$ such that $R / \ptor{R} \cong S$ are in $\mf{U}$?
\end{enumerate}
\end{question}
\noindent It would be interesting to answer these questions at least in some special cases, for example: for $S = W(\kk)$ (only the second question), for integral domains in general (more restrictively: for discrete valuation rings), for rings with one-dimensional tangent space. Thus, we pose the following test problems: 

\begin{question}
Which of the following rings are in $\mf{U}$ (are in $\mf{Q}$)?
\begin{itemize}
\item $W(\kk)[\sqrt[r]{p}\,]$
\item $W(\kk)[X]/(X^2-p^rX\,)$ 
\item $W(\kk)[[X]]/(p^rX)$
\end{itemize}
Here $r$ is an integer, $r>1$ in the first case and $r\geq 1$ in the other cases.
\end{question}

The only results of which the author is aware is that $\ZZ_5[\sqrt{5}]$ and $\ZZ_p[[X]]/(pX)$ for $p=3$ ($p \geq 3$) is in $\mf{U}$ (is in $\mf{Q}$). 

\subsection{Other remarks}

It is worth noting that Lemma~\ref{LemMaranda}, on which we based our arguments leading to the main result, can be applied also in a slightly different way. Not only in order to find some rings that are not in $\mf{U}$, but also in order to give a lower bound on the size of a group whose representation can realize $R$ as a universal deformation ring. More precisely:

\begin{lemma} \label{LemOtherUseOfMaranda}
Let $R \in \Ob(\cat)$ be given and suppose there exist $S \in \Ob(\cat)$, $r \in \NN$ and $f_1, f_2 \in \Hom_{\cat}(R,S)$ such that $\ptor{S}=0$, $f_1 \neq f_2$, and $f_1 \equiv f_2 \pmod{p^r \mf{m}_S}$. If $R$ is a universal deformation ring of a representation $\rhobar$ of a finite group $G$, then $p^{r+1} \mid \# G $. 
\end{lemma} 
\begin{proof}
Let $p^l$ be the largest power of $p$ dividing $\# G$. By definition of a universal deformation ring, morphisms $f_1$ and $f_2$ induce two different deformations of $\rhobar$ to $S$, so $f_1$ and $f_2$ are different modulo $p^l \mf{m}_S$ by Lemma~\ref{LemMaranda}. Using the assumption we conclude that $l > r$ and the claim follows.
\end{proof}

\begin{example}
Let $r \geq 1$ be an integer and suppose $R:= W(\kk)[X]/(X^2-p^rX)$  is a universal deformation ring of a representation of a finite group $G$. Since $\ptor{R}=0$ and we have $f_1 : R \xrightarrow{X \mapsto 0} R$, $f_2 : R \xrightarrow{X \mapsto p^r} R$ with the same reduction modulo $p^{r-1}\mf{m}_R$, Lemma~\ref{LemOtherUseOfMaranda} implies that $p^r \mid \# G$. 
\end{example}


\begin{thebibliography}{99}
\bibitem[Ati]{Atiyah} M.F. Atiyah, I.G. McDonald. Introduction to Commutative Algebra. Addison-Wesley Publishing Company, 1969.
\bibitem[Ble]{BleherFinite} F. M. Bleher, Universal deformation rings of group representations, with an application of Brauer's generalized decomposition numbers.  	Contemp. Math., 607, Amer. Math. Soc., Providence, RI, 2014, pp. 97-112.
\bibitem[BCdS]{BleherChinburgDeSmit} F. M. Bleher, T. Chinburg and B. de Smit. Inverse problems for deformation rings. Trans. Amer. Math. Soc. 365 (2013) 6149-6165.
\bibitem[Coh]{Cohen} I. S. Cohen, On the structure and ideal theory of complete local rings. Trans. AMS 59 (1946), 54-106.
\bibitem[Dor]{Dorobisz} K. Dorobisz. The inverse problem for universal deformation rings and the special linear group. Preprint, ArXiv:1308.1346v2.
\bibitem[EaM]{EardleyManoharmayum} T. Eardley, J. Manoharmayum, The inverse deformation problem. Preprint, ArXiv:1307.8356.
\bibitem[Eis]{Eisenbud} D. Eisenbud. Commutative algebra: with a view toward algebraic geometry. Springer, 1995. 
\bibitem[Gou]{Gouvea} F. Q. Gouvea. Deformations of galois representations. In B. Conrad and K. Rubin, editors, Arithmetic Algebraic Geometry, volume 9 of IAS/Park City Mathematics Series, pages 233--336. Amer. Math. Soc., Providence, RI, 2001.
\bibitem[Mar]{Maranda} J.-M. Maranda. On $\mf{p}$-adic integral representations of finite groups. Canad. J. Math. 5(1953), 344-355.
\bibitem[Maz]{Mazur2} B. Mazur, Deforming Galois representations. In: Galois Groups Over Q (Berkeley, CA, 1987), Math. Sci. Res. Inst. Publ. 16 Springer, New York, 385-437, 1989.
208-222.
\bibitem[Ste]{Stevenhagen} P. Stevenhagen, Local fields. Lecture notes available online at: \\{\tt 	http://websites.math.leidenuniv.nl/algebra/localfields.pdf}.
\end{thebibliography}
\end{document}